\documentclass[12pt]{article}

 \def\real{{\mathord{\mathbb R}}}
 \def\inte{{\mathord{\mathbb N}}}

\usepackage{amsmath}
\usepackage{amsthm}
\usepackage{amssymb}
\usepackage{graphicx}
\usepackage{comment}

\usepackage{latexsym,array,amsfonts,eucal}
\usepackage{mathrsfs,dsfont,bbm}

\usepackage{titlesec,fancyhdr,titletoc}
\usepackage{verbatim}

\allowdisplaybreaks

\begin{document}

\title{\huge \baselineskip1cm
Feynman-Kac formula for L\'evy processes and semiclassical (Euclidean) momentum representation}

\author{ \baselineskip0.1cm
Nicolas Privault\thanks{nprivault@ntu.edu.sg}
\vspace*{-0.3cm}
\\
School of Physical and Mathematical Sciences
\vspace*{-0.3cm}
\\
Nanyang Technological University
\vspace*{-0.3cm}
\\
21 Nanyang Link, Singapore 637371
              \and
Xiangfeng Yang\thanks{xyang@cii.fc.ul.pt}$\qquad$Jean-Claude Zambrini\thanks{zambrini@cii.fc.ul.pt}
\vspace*{-0.3cm}
\\
Grupo de F\'{i}sica Matem\'{a}tica
\vspace*{-0.3cm}
\\
Universidade de Lisboa
\vspace*{-0.3cm}
\\
Av. Prof. Gama Pinto 2, 1649-003 Lisboa, Portugal}

\baselineskip0.5cm

\maketitle

\baselineskip0.5cm

\vspace{-0.5cm}

\begin{abstract}
\baselineskip0.5cm
We prove a version of the Feynman-Kac formula for L\'evy processes and
integro-differential operators,
with application to the momentum representation of suitable quantum (Euclidean) systems
whose Hamiltonians involve L\'{e}vy-type potentials.
Large deviation techniques are used to obtain the limiting behavior of the systems as the Planck constant approaches zero. It turns out that the limiting behavior coincides with fresh aspects of the semiclassical limit of (Euclidean) quantum mechanics. Non-trivial examples of L\'{e}vy processes are considered as illustrations and precise asymptotics are given for the terms in both configuration and momentum representations.
\end{abstract}

\baselineskip0.5cm

\renewcommand{\theequation}{\thesection.\arabic{equation}}
\newtheorem{theorem}{Theorem}[section]
\newtheorem{thm}{Theorem}[section]
\newtheorem{remark}[thm]{Remark}
\newtheorem{prop}[thm]{Proposition}
\theoremstyle{definition}
\newtheorem{defn}[thm]{Definition}
\numberwithin{equation}{section}
\newcommand{\dx}{\Delta x}
\newcommand{\dt}{\Delta t}

\def\keywords{\vspace{.5em}\hspace{-2em}
{\textit{Keywords and phrases}:\,\relax%
}}
\def\endkeywords{\par}

\def\MSC{\vspace{.0em}\hspace{-2em}
{\textit{AMS 2010 subject classifications}:\,\relax%
}}
\def\endMSC{\par}

\baselineskip0.5cm

\keywords{L\'{e}vy process, Feynman-Kac type formula, momentum representation, large deviations}

\MSC{Primary 60J75, 60G51; secondary 60F10, 47D06}

\baselineskip0.7cm

\section{Introduction}\label{sec:introduction}
 Consider the L\'{e}vy-type potential
\begin{equation}\label{Levy-Type-Potential}
V(x)=ibx+\frac{1}{2}\sigma^2x^2-\int_{\mathbb{R}\setminus\{0\}}
 ( e^{-ixk}-1+ixk {\bf 1}_{\{|k|\leq1\}} ) \nu(dk),
 \qquad
 x\in\mathbb{R},
\end{equation}
 where $b,\sigma\in\mathbb{R}$ and the L\'{e}vy measure $\nu$
 satisfies $\int_{\mathbb{R}\setminus\{0\}}(1\wedge k^2)\nu(dk)<\infty$,
 and the pseudo-differential operator $V(i \nabla)$ given by
\begin{align}\label{box-cross}
V(i \nabla)u(p)=- b u'(p)-\frac{\sigma^2}{2}u''(p)-\int_{\mathbb{R}\setminus\{0\}}
 ( u(p+ k)-u(p)- k u'(p) {\bf 1}_{\{|k|\leq1\}} ) \nu(dk),
\end{align}
 cf. \cite{jacobschilling} and the references therein.
 In this paper we consider the partial differential equation
\begin{align}\label{eq:momen-equation}
\begin{cases}
\displaystyle
\frac{\partial u}{\partial t}(t,p)=-U(p)u(t,p)-V(i\nabla)u(t,p),& (t,p)\in (0,\infty)\times \mathbb{R}\\
u(0,p)=g(p),&p\in\mathbb{R}
\end{cases}
\end{align}
 with prescribed initial condition $g(p)$
 and we derive the Feynman-Kac type formula
\begin{align}\label{eq:Feyn-Kac-formula}
 u (t,p)=\mathbb{E} \left[
 g(\xi_t )\exp\left(
 -\int_0^tU(\xi_s)\,\,ds\right) \,\Big|\,\xi_0 =p\right],
\end{align}
 for the solution $u(t,p)$ of \eqref{eq:momen-equation},
 cf. Theorem~\ref{th:Feynman-Kac-for-Levy} in Section~\ref{sec:Feynman-Kac}.
\\

 Here, $( \xi_t)_{t\in \real_+}$ is a real-valued
 L\'{e}vy process on $(\Omega,\mathcal{F},\mathbb{P})$,
 whose characteristic function
 is given by the L\'{e}vy-Khintchine formula
\begin{align}\label{levy-characteristic-exponent}
\mathbb{E} \left[\exp\left( - i x\cdot \xi_t\right) \right]=\exp\left(
 -
 t V(x)\right),
\end{align}
 with
$$
 V(x)=ibx+\frac{\sigma^2}{2}x^2-\int_{\mathbb{R}\setminus\{0\}}
 ( e^{-ix k}-1 ) \nu(dk),
$$
 where
 $\int_{\mathbb{R}\setminus\{0\}}(1\wedge k^2)\nu(dk)<\infty$
 and the L\'{e}vy measure $\nu(dk)$ will be chosen symmetric
 under time reversal $k\to -k$ so that \eqref{box-cross} reduces to
\begin{align*}
V(i \nabla)u(p)=- b u'(p)-\frac{\sigma^2}{2}u''(p)-\int_{\mathbb{R}\setminus\{0\}}
 ( u(p+ k)-u(p) ) \nu(dk).
\end{align*}
 The function $U(p)$ in \eqref{eq:momen-equation} is sometimes called \textit{killing rate} (see Section 3.5 in \cite{applebk}). We shall, however, avoid the term `killing' which is inappropriate to our time reversible framework, and refer simply to the `rate'.
\\

 Note that formulas such as \eqref{eq:Feyn-Kac-formula}
 have been implicitly used in the literature,
 for example for the construction of subsolutions
 and the derivation of blow-up criteria for
 semilinear PDEs of the form
$$
\begin{cases}
\displaystyle
\frac{\partial u}{\partial t}(t,p)=
 u^{1+\beta} (t,p)
-V(i\nabla)u(t,p),& (t,p)\in (0,\infty)\times \mathbb{R}\\
u(0,p)=g(p),&p\in\mathbb{R},
\end{cases}
$$
 where $g$ is a nonnegative function, $\beta >0$,
 see e.g. \cite{blmw}, \cite{prlm2}, when
 $V(i \nabla) = - ( - \Delta)^{\alpha/2}$ is a
 fractional Laplacian power, $0<\alpha\le2$,
 and \cite{prlm} for the generator
$$
 V(i \nabla)u(p)= - \int_{\mathbb{R}\setminus\{0\}}
 ( u(p+ k)-u(p)) \frac{e^{-k}}{k} dk,
$$
 of the standard gamma process.
\\

 In this paper, based on Feynman and Hibbs' suggestion
 that in the momentum representation of
 (Euclidean) quantum physics, the underlying
 stochastic processes should belong to some class of time reversible
 jump processes (cf. \cite{Feyn2}),
 we apply our result to the
 momentum representation
\begin{equation}
\label{p2}
 \hat{H}u(p)=\mathcal{F}H\mathcal{F}^{-1} u(p)
 \hat{H}u(p)=\frac{p^2}{2}u(p)+V(i\hbar \nabla)u(p),
\end{equation}
 of the original
 (configuration representation) Hamiltonian
$$
 H=-\frac{\hbar^2}{2} \triangle+V,
$$
 where $\hbar$ is the Planck constant
 and $\mathcal{F}$ is the Fourier transform
$$\mathcal{F}u(p)=\frac{1}{\sqrt{2\pi\hbar}}\int_{\mathbb{R}}e^{-ipx/\hbar}u(x)dx
,
$$
 for any rapidly decreasing function $u(\cdot)\in S(\mathbb{R})$,
 with inverse
$$\mathcal{F}^{-1}v(x)=\frac{1}{\sqrt{2\pi\hbar}}\int_{\mathbb{R}}e^{ipx/\hbar}v(p)dp,\qquad v(\cdot)\in S(\mathbb{R}).$$
 The corresponding
 Euclidean version of the associated Schr\"{o}dinger equation in
 momentum representation (in the sense of \cite{prz2} and \cite{prz})
 reads
\begin{align}\label{eq:momen-original}
-\hbar \frac{\partial \hat{\eta}^*}{\partial t}(t,p)=\hat{H}\hat{\eta}^*(t,p),
\end{align}
 which is one of the two adjoint equations whose positive solutions are needed to produce quantum-like probability measures in \cite{prz2} and \cite{prz}.
\\

 In addition we will consider a general rate function $U(p)$
 instead of $p^2/2$ in \eqref{p2} and deal with the equation
\begin{align} \label{eq:momen-equation2}
\begin{cases}
\displaystyle
\hbar \frac{\partial \eta^*}{\partial t}(t,p)
 =
 -U(p)\eta^*(t,p)-V(i\hbar \nabla)\eta^*(t,p),&
 (t,p)\in (0,\infty)\times \mathbb{R}
\\
 \eta^*(0,p)=g(p),&p\in\mathbb{R},
\end{cases}
\end{align}
 instead of \eqref{eq:momen-original} with
 initial condition $g(\cdot)\in S(\mathbb{R})$.
 Finally we will study the limiting behavior of \eqref{eq:momen-equation2}
 as $\hbar$ approaches $0$.
\\

 Usually there are two different methods to derive a Feynman-Kac type formula:
\begin{itemize}
\item martingale methods combined with
 the It\^{o} formula (see for example Section 5.7 of
 \cite{Karatzas-Shreve-1991} and Section 3.19 of
\cite{Rogers-Williams-2000}), or
\item semigroup methods (see Section 6.7.2 in \cite{applebk}, Section 3.19 in \cite{Rogers-Williams-2000} and Section 1.5 of \cite{Freidlin-Wentzell-1998}).
\end{itemize}
 The first method is generally applicable under the main assumption that the partial integro-differential equation \eqref{eq:momen-equation} has a solution with suitable growth assumption (see Proposition 4 in \cite{Nualart-Schoutens-2001} and Theorem 4.1 in \cite{Chan-1999} for more details by using stochastic integrals driven by L\'{e}vy processes).
 However, an unbounded rate $U(p)$ function
 is not compatible with the conditions of \cite{Nualart-Schoutens-2001} and \cite{Chan-1999}. In this paper, we use the second (semigroup) method to prove the existence (and even the uniqueness) of a solution for \eqref{eq:momen-equation} and to induce a Feynman-Kac type formula for the solution simultaneously. It is also known that the semigroup method (with infinitesimal generator) only works for (at least) bounded rates. In order to deal with the unbounded rate, in this paper we use this method in a weaker form (explained at the beginning of Section \ref{sec:Feynman-Kac}) which is explicitly presented in Theorem \ref{th:Feynman-Kac-for-Levy} and Theorem \ref{th:Feynman-Kac-NoIndepIncream}.
\\

 In addition we also study the limiting behaviors of the solutions
given by Feynman-Kac type formulas in terms of large deviations.
In particular we show in detail that these limiting behaviors have exactly
the same patterns as the semiclassical limits of (Euclidean)
quantum mechanics in several special cases, and we derive precise asymptotics as well for the drift terms in both configuration and momentum representations.
\\

This paper is organized as follows.
In Section~\ref{sec:Feynman-Kac} we present the Feyman-Kac
formula \eqref{eq:Feyn-Kac-formula} and its proof.
The limiting behaviors of the solution given by Feynman-Kac type formula are
then studied in terms of large deviations in Section \ref{sec:limit-behavior}.
\section{Feynman-Kac type formulas}\label{sec:Feynman-Kac}
Let $A$ be the infinitesimal generator of the semigroup
$$
 T_tf(x) 
 :=\mathbb{E}[f(\xi_t)|\xi_0=x],
$$
 of a time-homogeneous Markov process $\xi_t$
 where $f$ is in a space $\mathbf{B}$ of bounded measurable
 functions with the uniform norm, and such that
$$
 \lim_{t\to 0^+}\sup_x\left| \frac{T_tf(x)-f(x)}{t} -Af(x)\right|=0,
 \qquad
 f \in \mathbf{B}.
$$
 For every $f$ in the domain of $A$,
 the function 
$$
 u(t,x):=T_tf(x)=\mathbb{E}_x [ f(\xi_t) ] 
$$
 is the unique (in the sense of boundedness) solution of
$$
 \frac{\partial u}{\partial t}(t,x)=Au(t,x)\,\, \text{ with } u(0,x)=f(x).
$$
 The main ideas for the derivation of
 Feynman-Kac type formulas by semigroup methods
 for differential equations of the form
\begin{equation}
\label{dkolddd}
 \frac{\partial v}{\partial t}(t,x)=Av(t,x)-c(x)v(t,x),\,\,\,\, v(0,x)=f(x)
\end{equation}
 with rate $c(x)$, can be formulated as follows,
 see Section 3.19 of \cite{Rogers-Williams-2000} and
 Section 1.5 of \cite{Freidlin-Wentzell-1998} for details.
\\

 The idea of the semigroup method is to show that
$$
 \widetilde{A}f(x):=Af(x)-c(x)f(x)
$$
 is the infinitesimal generator of the new semigroup
$$
 \widetilde{T}_tf(x):=\mathbb{E}_x\left[
 f(\xi_t)\exp \left( -\int_0^tc(\xi_s)ds \right) \right],
$$
 i.e. $v(t,x):=\widetilde{T}_tf(x)$ is a solution of \eqref{dkolddd}.
 This only holds under (strong) assumptions,
 including the boundedness of $c(x)$,
 cf. \cite{Rogers-Williams-2000}.
 Note that in this case the derivative
$$\frac{\partial v}{\partial t}(t,x)=Av(t,x)-c(x)v(t,x)$$
 is uniform in $x,$ in the sense that
\begin{align}\label{eq:uniform-conv}
\lim_{h\to 0}\sup_x\left|
\frac{v(t+h,x)-v(t,x)}{h}-(Av(t,x)-c(x)v(t,x))\right|=0,
\end{align}
 for every fixed $t>0$.
 However, pointwise convergence is sufficient in \eqref{eq:uniform-conv}
 for the analysis of a differential equation,
 and we choose to
 weaken the boundedness condition on $c(x)$ when deriving our
 Feynman-Kac formulas.
\\

 In the sequel, $( \xi_t )_{t\geq0}$
 will be a L\'{e}vy process
 defined on $(\Omega,\mathcal{F},\mathbb{P})$
 whose characteristic exponent is given by
 \eqref{levy-characteristic-exponent},
 and which is Markov
 with respect to the increasing filtration $\mathcal{P}_t=\sigma(\xi_s,0\leq s\leq t)$.
 Furthermore, $( \xi_t )_{t\geq0}$ has a modification (still denoted by $( \xi_t )_{t\geq0}$) which is
 right continuous with left limits. We say that the process $\xi$ has \textit{bounded jumps} if there is $C>0$ such that
$$\sup_{t\geq 0}\left|\xi_t-\xi_{t^-}\right|<C\quad\text{a.s.}$$
 It follows from Corollary 2.4.9 of \cite{applebk} that $\xi$ has bounded jumps if and only if $\nu$ has a bounded support.
\\

 We denote by
$$
S(\mathbb{R},\mathbb{C}) :=\left\{f(\cdot):\mathbb{R}\rightarrow\mathbb{C} \text{ with }\sup_x\left|x^mf^{(n)}(x)\right|<\infty,
 \quad
 \forall \ m,n \in \inte
 \right\},
$$
 the Schwartz space
 of all infinitely differentiable rapidly decreasing functions
 with bounded $n$-th derivatives $f^{(n)}(x)$
 on the real line $\mathbb{R}$ to the complex plane $\mathbb{C}$,
 with $S(\mathbb{R}) = S(\mathbb{R},\mathbb{R})$.
\begin{theorem}\label{th:Feynman-Kac-for-Levy}
 Assume that the rate function $U(p)$ is smooth and
 such that
\begin{equation}
\label{jkld}
 U(p)\geq c|p|^a,
 \qquad
 |p|>C,
\end{equation}
 for some $c>0,a\geq1$ and $C>0$,
 and that its derivatives satisfy
\begin{equation}
\label{jkld.1}
 \left|U^{(n)}(p)\right|\leq c(n)\left(1+|p|^{k(n)}\right), \qquad |p|>C,
\end{equation}
 for some $c(n), k(n) >0$.
 Then \eqref{eq:Feyn-Kac-formula} is a solution to \eqref{eq:momen-equation}
 provided $\nu$ has moments of all orders
 and $g\in S(\mathbb{R})$.
\end{theorem}
\begin{proof}
 Throughout this proof we take $U(p) = p^2/2$
 without loss of generality.
 For any bounded measurable $f(x)$, let
 $T_tf(x)=\mathbb{E}[f(\xi_t+x) ]$
 denote the semigroup associated with the L\'evy process
 $\xi_t$
 started at $\xi_0=0$.
 For any function $g\in S(\mathbb{R}),$ the infinitesimal generator $A$ of
 $\xi_t$ is
\begin{align*}
Ag(x)=b g'(x)+\frac{\sigma^2}{2}g''(x)+ \int_{\mathbb{R}\setminus\{0\}}\left(
 g(x+ k)-g(x)\right) \nu(dk),
\end{align*}
 see Theorem 3.3.3 in \cite{applebk}.
 We also define
$$
 \widetilde{T}_tg(x)=\mathbb{E}\left[g(\xi_t)\exp \left( -\frac{1}{2} \int_0^t | \xi_s |^2 \,\,ds \right)\Big| \xi_0=x\right]$$
 and let $\widetilde{A}g(x): = Ag(x) - x^2g(x)/2$.
\\

\textit{Step 1:} We show for each fixed $x\in\mathbb{R}$ and $g\in S(\mathbb{R})$ that
\begin{align}\label{eq:aim-to}
\lim_{t\to 0^+}t^{-1}\left(\widetilde{T}_tg(x)-g(x)\right)=\widetilde{A}g(x).
\end{align}
 By the relation
\begin{align}\label{martingale-approach}
\widetilde{T}_tg(x)-g(x)=\int_0^t\widetilde{T}_s\widetilde{A}g(x)ds
\end{align}
 which follows from martingale arguments,
 cf. e.g. Section~3.19 of \cite{Rogers-Williams-2000},
 we get
\begin{align*}
\lim_{t\to 0^+}t^{-1}&\left(\widetilde{T}_tg(x)-g(x)\right)=\lim_{t\to 0^+}t^{-1}\int_0^t\widetilde{T}_s\widetilde{A}g(x)\\
&=\lim_{t\to 0^+}t^{-1}\int_0^t\mathbb{E}\left[\left(Ag(\xi_s)-\frac{1}{2} |\xi_s |^2g(\xi_s)\right)\exp\left(-\frac{1}{2} \int_0^s |\xi_u|^2\,\,du\right)\,\,\Big|\,\, \xi_0=x\right]ds\\
&=\lim_{t\to 0^+}\int_0^1\mathbb{E}\left[\left(Ag(\xi_{t\theta})-\frac{1}{2} |\xi_{t\theta} |^2g(\xi_{t\theta})\right)\exp\left( -\frac{1}{2} \int_0^{t\theta} |\xi_u |^2\,\,du\right) \,\,\Big|\,\, \xi_0 =x\right]d{\theta}.
\end{align*}
 The first term is easily to be computed as
\begin{align*}
\lim_{t\to 0^+}\int_0^1\mathbb{E} \left[Ag(\xi_{t\theta} )\exp\left( -\frac{1}{2} \int_0^{t\theta} | \xi_u |^2\,\,du\right)\,\,\Big|\,\, \xi_0=x\right]d{\theta}=Ag(x)
\end{align*}
since $Ag(x)$ is bounded and continuous in $x.$ The second limit
\begin{align}\label{temp-8-8}
\lim_{t\to 0^+}\int_0^1\mathbb{E}\left[\xi_{t\theta} ^2\cdot g(\xi_{t\theta})\exp\left( -\frac{1}{2} \int_0^{t\theta} |\xi_u |^2\,\,du\right) \,\,\Big|\,\, \xi_0 =x\right]d{\theta}=x^2g(x)
\end{align}
since $x^2g(x)$ is bounded and continuous in $x$ which is from the fact that $g\in S(\mathbb{R}).$ However, here we present another proof of (\ref{temp-8-8}) which works for any (only) bounded and continuous $g.$ The reason why we give this proof is that, in Theorem \ref{th:Feynman-Kac-NoIndepIncream} below, the initial $g$ is only assumed to be bounded and continuous, and such a proof is needed there. Under the existence of moments of all orders for $\nu$
 we have the moment bounds
\begin{align}\label{eq:finite-moments}\sup_{t\in[0,1]}\mathbb{E}\left[(\xi_{t})^m\,|\, \xi_0=x\right]<\infty, \qquad
 m \geq 1,
\end{align}
 cf. e.g. \cite{bassan} and Relations~(1.1)-(1.2) in \cite{momentpoi}.
 We write
\begin{align*}
&\lim_{t\to 0^+}\int_0^1\mathbb{E}\left[ |\xi_{t\theta}|^2g(\xi_{t\theta})\exp\left( -\frac{1}{2} \int_0^{t\theta}|\xi_u|^2\,\,du\right)\,\,\Big|\,\, \xi_0=x\right]d{\theta}\\
&=\lim_{t\to 0^+}\int_0^1\mathbb{E}\left[ |\xi_{t\theta}|^2g(\xi_{t\theta})\,\,\Big|\,\,\xi_0=x\right]d{\theta}\\
&\quad+\lim_{t\to 0^+}\int_0^1\mathbb{E}\left[ |\xi_{t\theta}|^2g(\xi_{t\theta})\left(\exp\left(-\frac{1}{2} \int_0^{t\theta}|\xi_u|^2\,\,du\right)-1\right)\,\,\Big|\,\, \xi_0=x\right]d{\theta}\\
& = : I+II,
\end{align*}
 with $I= x^2g(x)$ because of \eqref{eq:finite-moments}.
 The second component vanishes by noting that
\begin{align*}
\mathbb{E}\left[\left(\exp\left( -\frac{1}{2} \int_0^{t\theta}|\xi_u|^2\,\,du\right) -1\right)^2\,\,\Big|\,\, \xi_0=x\right]d{\theta}=o(t).
\end{align*}

\textit{Step 2:} We show that
 $\widetilde{T}_tg(\cdot)\in S(\mathbb{R})$ for all
 $g\in S(\mathbb{R})$ and $t>0$.
 From the independence and stationarity of
 increments of $\xi_t$ we have
\begin{align*}
\widetilde{T}_tg(x)&=\mathbb{E}\left[g(\xi_t)\exp\left( -\frac{1}{2} \int_0^t|\xi_s|^2\,\,ds\right)\,\,\Big|\,\, \xi_0=x\right]\\
&=\mathbb{E}\left[g(\xi_t+x)\exp\left( -\frac{1}{2} \int_0^t|\xi_s+x|^2\,\,ds\right) \right],
\end{align*}
 hence by the bound $\mathbb{E}[ | \xi_t |^{k}]<\infty$, $k \geq 1$,
 the $n$-th partial derivatives $\partial^n \widetilde{T}_tg(x)/\partial x^n$ can be bounded as
\begin{align}\label{partial-bound}
\left(\frac{\partial^n \widetilde{T}_tg(x)}{\partial x^n}\right)^2\leq c(n)(1+x^{2k(n)})\mathbb{E}\left[\exp\left( -\frac{1}{2} \int_0^t|\xi_s+x|^2\,\,ds\right) \right]
\end{align}
for some positive constant $c(n)$ depending on $n$ and an integer $k(n)$ depending on $n$.
 For any positive integer $m,$ we analyze the part with expectation in \eqref{partial-bound} as follows
\begin{equation}\label{convex-use}
\begin{aligned}
&\mathbb{E}\left[\exp\left( -\frac{1}{2} \int_0^t|\xi_s+x|^2\,\,ds\right) \right]\leq\mathbb{E}\left[\exp\left( -\frac{t}{2}\left(x+\frac{1}{t}\int_0^t\xi_s\,ds\right)^2\right) \right]\\
&=\mathbb{E}\left[ {\bf 1}_{\left\{\omega: \left| \int_0^t\xi_s\,ds\right|\leq t |x|^{1/2}\right\}}\exp\left( -\frac{t}{2}\left(x+\frac{1}{t}\int_0^t\xi_s\,ds\right)^2\right) \right]\\
&\,\,\,\,+\mathbb{E}\left[ {\bf 1}_{\left\{\omega: \left|
 \int_0^t\xi_s\,ds\right|> t |x|^{1/2}\right\}}\exp\left( -\frac{t}{2}\left(x+\frac{1}{t}\int_0^t\xi_s \,ds\right)^2\right) \right] = : \ell_1+\ell_2.
\end{aligned}
\end{equation}
For $\ell_2,$ it follows from Tchebychev type estimates that
\begin{align*}
\ell_2\leq\mathbb{P} \left(
 \left\{\omega: \left|\frac{1}{t}\int_0^t\xi_s \,ds\right|>|x|^{1/2}\right\}
 \right)
 \leq
 \frac{1}{|x|^m t^{2m}}
 \mathbb{E} \left[\left( \int_0^t\xi_s \,ds\right)^{2m}\right].
\end{align*}
Thus $\sup_x|x|^m\ell_2<\infty$ holds since $\mathbb{E} [(\xi_t )^{k}]<\infty.$ It is easy to see that for large enough $|x|$ we have
$\ell_1\leq e^{ -t|x|/2}$,
which yields
$$\sup_x\left|x^m\frac{\partial^n \widetilde{T}_tg(x)}{\partial x^n}\right|<\infty.$$

\textit{Step 3:} In this step, we replace $g(x)$ in \eqref{eq:aim-to} by $\widetilde{T}_tg(x)$ to get
$$\frac{\partial^+ \widetilde{T}_tg(x)}{\partial t}=\widetilde{A}\widetilde{T}_tg(x),\quad \widetilde{T}_0g(x)=g(x).$$
In this step, we show that the right-hand derivative can be replaced by the two-sided derivative $\partial \widetilde{T}_tg(x)/\partial t.$ To this end, we just need to show that the right-hand derivative $\widetilde{A}\widetilde{T}_tg(x)$ is continuous in $t.$ Let us recall
$$\widetilde{A}\widetilde{T}_tg(x)=A\widetilde{T}_tg(x)-\frac{1}{2}x^2\widetilde{T}_tg(x).$$
The continuity in $t$ is then easily from the definitions of $A$ and $\widetilde{T}_t$ with the help of \eqref{martingale-approach} repeatedly.
\end{proof}
 Note that
 in Theorem \ref{th:Feynman-Kac-for-Levy} the condition \eqref{jkld} (i.e. $U(p)\geq c|p|^a$ for some $a\geq1$) is necessary
 in general
 since in the first inequality of \eqref{convex-use} we
 need the convexity of $c|p|^a$ in $p.$
 However, condition \eqref{jkld} can be weakened to any $a>0$
 by assuming in addition that the process
 $( \xi_t )_{t\geq0}$ is a subordinator, i.e.
 a one-dimensional a.s. non-decreasing
 L\'{e}vy process, cf. Section 1.3.2 in \cite{applebk},
 which will remain a.s. positive provided $\xi_0 \geq a_*>0.$
 For instance, Theorem~\ref{th:Feynman-Kac-for-Levy.1} applies in
 case $U(p)=p^{1/2}.$
\begin{theorem}
\label{th:Feynman-Kac-for-Levy.1}
 Assume that $( \xi_t )_{t\geq0}$ is a
 subordinator and that in addition to
 \eqref{jkld.1}
 the smooth rate function $U(p)$
 satisfies
$$
 U(p)\geq c|p|^a,
 \qquad
 |p|>C,
$$
 for some $c>0,a >0$ and $C>0$.
 Then \eqref{eq:Feyn-Kac-formula} is a solution to \eqref{eq:momen-equation}
 provided $\nu$ has moments of all orders
 and $g\in S(\mathbb{R})$.
\end{theorem}
\begin{proof}
 We adjust \eqref{convex-use}
 to make the arguments go through for the new
 (possibly not convex) function $cp^a$.
 Namely we write
\begin{equation}\label{convex-use-2}
\begin{aligned}
&\mathbb{E}\left[\exp\left( -\int_0^tU(\xi_s+x)\,\,ds\right) \right]\leq\mathbb{E}\left[\exp\left( - c \int_0^t\left(\xi_s+x\right)^a\,ds\right) \right]\\
&=\mathbb{E}\left[ {\bf 1}_{\left\{\omega: \,\,\sup_{0\leq s\leq t}\xi_s\leq x^{1/2}\right\}}\exp\left( - c \int_0^t\left(\xi_s+x\right)^a\,ds\right) \right]\\
&\,\,\,\,+\mathbb{E} \left[ {\bf 1}_{\left\{\omega: \,\,\sup_{0\leq s\leq t}\xi_s > x^{1/2}\right\}}\exp\left( - c \int_0^t\left(\xi_s +x\right)^a\,ds\right) \right]:=\ell_1+\ell_2,
 \quad
 x>0,
\end{aligned}
\end{equation}
 with $\ell_1\leq\exp\left( - c t x^{a/2}\right)$
 for large $x$ as above.
 The second term is estimated as $\ell_2\leq\mathbb{P}\left(
 \omega: \,\,\sup_{0\leq s\leq t}\xi_s > x^{1/2}\right)
 =\mathbb{P}
 \left(
 \omega: \,\,\xi_t > x^{1/2}\right),$
 then Tchebychev type estimates complete the argument.
\end{proof}
 The condition $g\in S(\mathbb{R})$ in Theorem \ref{th:Feynman-Kac-for-Levy}
 is somewhat restrictive and
 can also be relaxed into assuming only continuity and boundedness of $g$
 if we restrict ourselves to a special family
 $( \zeta_t )_{t\geq0}$ of L\'{e}vy processes which
 are real-valued pure jump processes with infinitesimal generator
$$\mathcal{A}f(p)= \int_{\mathbb{R}} ( f(p+ k)-f(p) ) \mu(dk)$$
 for bounded $f(\cdot):\mathbb{R}\rightarrow \mathbb{R},$
 where $\mu(\mathbb{R})<\infty$ and $\mu(dk)$ is symmetric under $k\to -k$, cf. e.g. Section 4.2 of \cite{Ethier-Thomas-2005}.
 In this setting we consider
 the partial integro-differential equation
\begin{align}\label{eq:momen-equation-no-indep-incre}
\begin{cases}
\displaystyle
\frac{\partial u_*}{\partial t}(t,p)=-U(p) u_*(t,p)+\int_{\mathbb{R}} 
 ( u_*(t,p+ k)-u_*(t,p) ) 
 \mu(dk),& (t,p)\in (0,\infty)\times \mathbb{R}\\
u_*(0,p)=g(p), \quad p\in\mathbb{R}, &
\end{cases}
\end{align}
where $U(p)$ is continuous and satisfies $c_1\leq U(p)\leq c_2(1+|p|^M)$ for some $c_1\in\mathbb{R},c_2>0$ and $M>0.$
 This allows us in particular
 to work in the momentum representation of (Euclidean) quantum mechanics
 with a general potential $U(p)$ instead of only $p^2/2$ as
 in \cite{prz}.
\begin{theorem}\label{th:Feynman-Kac-NoIndepIncream}
 Suppose that $\mu(dk)$ is symmetric under $k\to -k$,
 $\mu(\mathbb{R})<\infty$, and
 $U(p)$ is continuous and satisfies
$$
 c_1\leq U(p)\leq c_2(1+|p|^M)
$$
 for some $c_1\in\mathbb{R},c_2>0$ and $M>0.$ If the initial condition $g$ of \eqref{eq:momen-equation-no-indep-incre} is continuous and bounded, and $\int_\mathbb{R}|k|^{2M}\mu(dk)<\infty,$ then
\begin{align}\label{eq:Feyn-Kac-formula-on-indep-incre}
u_*(t,p)=\mathbb{E} \left[g(\zeta_t )\exp\left(
 - \int_0^tU(\zeta_s )\,\,ds\right) \,\Big|\,\zeta_0 =p\right]
\end{align}
is a solution to \eqref{eq:momen-equation-no-indep-incre}.
\end{theorem}
\begin{proof} 
The proof is quite similar to that of Theorem \ref{th:Feynman-Kac-for-Levy} by
 replacing $S(\mathbb{R})$
 with the space of bounded measurable functions on $\mathbb{R}$.
 Since $\mu (\real )$ is finite
 the finiteness of the moment
 $\sup_{t\in[0,1]}\mathbb{E} [ (\zeta_t )^{2M} ]$
 follows directly
 from the assumption
 $\int_\mathbb{R}|k|^{2M}\mu(dk)<\infty$
 by e.g. \cite{bassan} or
 Relations~(1.1)-(1.2) in \cite{momentpoi}. In Steps 2 and 3 we need the boundedness and the continuity in $x$ of the new semigroup $$\widetilde{T}_tg(p)=\mathbb{E} \left[g(\zeta_t )\exp\left(- \int_0^tU(\zeta_s )\,\,ds\right) \,\,\Big|\,\, \zeta_0 =p\right]$$
which follow from the boundedness and the continuity of $g(x).$
\end{proof}
There are a number of ways to prove that (\ref{eq:Feyn-Kac-formula-on-indep-incre}) is the unique (bounded) solution to \eqref{eq:momen-equation-no-indep-incre} under additional appropriate assumptions. For instance, if the rate $U$ is bounded, then the Gronwall's lemma implies the uniqueness.

\section{Limiting behaviors}\label{sec:limit-behavior}
 This section is concerned with the limiting behavior of
 solutions and drift terms as $\hbar$ tends to $0$.
 We start in Section~\ref{subsec:an-illustration}
 with an illustration
 of how large deviations (of integral forms) can be applied
 in the case of pure jump processes.
 Detailed formulations are then presented together
 with some close connections with the semiclassical limits discussed in \cite{Kolsrud-Zambrini-1991}.
 Finally we provide
 precise asymptotics for the drift terms in both configuration and momentum representations of Euclidean quantum physics.
\subsection{An illustration with pure jump processes}\label{subsec:an-illustration}
 To illustrate the analysis of
 limiting behaviors of solutions defined through
 \eqref{eq:Feyn-Kac-formula} and \eqref{eq:Feyn-Kac-formula-on-indep-incre}
 we consider the limit as $\hbar$ tends to $0$ of the solution
\begin{align*}
u^{\hbar}(t,p)=\mathbb{E}^{\hbar}\left[\exp\left(-\hbar^{-1}\int_0^tU(\zeta_s^{\hbar})\,\,ds\right) \,\Big|\,\zeta_0^{\hbar}=p\right].
\end{align*}
 to \eqref{eq:momen-equation-no-indep-incre} when $g\equiv1$
 and $( \zeta_t^{\hbar} )_{t\geq0}$ is the
 pure jump processes with infinitesimal generator
$$\mathcal{A}f(p)=\hbar^{-1}\int_{\mathbb{R}} (
 f(p+\hbar k)-f(p) ) \mu(dk).$$
Besides the time symmetry assumption and $\mu(\mathbb{R})<\infty,$ we further impose the bounded support condition
$$\mu([-N,N]^c)=0,\,\,\,\mu([-N,-N+\epsilon])>0,\,\,\,\mu([N-\epsilon,N])>0,\,\,\,\forall \epsilon>0,
$$
 on the measure $\mu$, for some $N>0.$
The trajectories of $( \zeta_s^{\hbar} )_{0\leq s\leq t}$ belong to the space $D( [0,t] )$ of all right continuous functions with left limits (equipped with the uniform norm). It has been shown in \cite{Wentzell-LD-Markov-Processes-1986} (see Theorems 3.2.1, 3.2.2 and 4.1.1 therein) that the family
 $( \zeta_s^{\hbar} )_{0\leq s\leq t}$ satisfies a \textit{large deviation principle} over $D( [0,t])$ with the action functional $S(\phi):=\int_0^tL_0(\phi'(s))ds$ for absolutely continuous function $\phi(t)$ where
$$L_0(u)=\sup_{x\in\mathbb{R}} ( xu-H_0(x) ),
 \quad H_0(x)=\int_\mathbb{R}(e^{xk}-1)\mu(dk).$$
 More precisely, for any measurable set $\Gamma\subseteq D ( [0,t] )$
 we have
\begin{equation}\label{eq:LDP-definition}
-\inf_{\phi\in\Gamma^o}S(\phi)\leq \liminf_{\hbar\to 0}\hbar\ln \mathbb{P}^{\hbar} ( \zeta^\hbar\in \Gamma )
 \leq \limsup_{\hbar\to 0}\hbar\ln \mathbb{P}^{\hbar} (
 \zeta^\hbar\in \Gamma ) \leq -\inf_{\phi\in\bar{\Gamma}}S(\phi)
\end{equation}
where $\Gamma^o$ (resp. $\bar{\Gamma}$) is the interior (resp. closure) of $\Gamma.$
\\

It is proved from \eqref{eq:LDP-definition} through Varadhan's integral lemma (see Section 4.3 \cite{Dembo-Zei}) that
\begin{eqnarray*}
\lefteqn{ 
\lim_{\hbar\to 0}\hbar\ln \mathbb{E}^{\hbar}\left[\exp\left( -\hbar^{-1}\int_0^tU(\zeta_s^{\hbar})\,\,ds\right)\,\Big|\,\zeta_0^{\hbar}=p\right]
} 
\\ 
 &=& \sup_{\phi\in D_0 ([0,t])}\left(
 -\int_0^tU(\phi(s)+p)\,\,ds-S(\phi)\right)
\\
&= & \sup_{\phi\in D_0 ([0,t]) }\left( -\int_0^tL(\phi'(s),\phi(s))ds\right),
\end{eqnarray*}
where $L(\phi'(s),\phi(s))=L_0(\phi'(s))+U(\phi(s)+p)$ and $D_0 ( [0,t] )$ is the subspace of $D ( [0,t] ) $ with initial position $0.$ For general $g$ satisfying suitable boundedness and smoothness conditions, we may also deduce the
 asymptotic expansion
\begin{equation*}
 u^{\hbar}_*(t,p)
 =
 \exp\left( h^{-1}\sup_{\phi\in D_0 ([0,t])}
 \left(
 -\int_0^tL(\phi'(s),\phi(s))\,\,ds\right)
 \right) \left(
 \sum_{0\leq i\leq n}k_i\hbar^{i/2}+o(\hbar^{n/2})\right)
\end{equation*}
 as $\hbar\to 0$,
by means of precise asymptotics for large deviations which have been established in \cite{Yang-Annals-2012+} and Chapter 5 of \cite{Wentzell-LD-Markov-Processes-1986}.
\subsection{Harmonic oscillator Hamiltonian}
 In this section we turn to \eqref{eq:Feyn-Kac-formula}
 and take
$$U(p)=p^2/2, \quad
 b=\nu=0, \quad \mbox{and} \quad \sigma^2=1
$$
 in the characteristic exponent \eqref{levy-characteristic-exponent},
 i.e. $( \xi^{\hbar}_t )_{t\in \real_+} =\sqrt{\hbar}
 (W_t )_{t\in \real_+}$, where $(W_t )_{t\in \real_+}$ is a
 one-dimensional Wiener process on the real line.
\\

 In order to show the connection
 between large deviations and semiclassical
 limit of (Euclidean) quantum mechanics
 we consider
 $\bar{u}^{\hbar}(t,p):=u^{\hbar}_*(1-t,p)$ for $t\in[0,1]$.
 It follows from \eqref{eq:momen-equation} with
 initial $g(p)=\exp \left( -p^2/\hbar \right)$ that the positive function $\bar{u}^{\hbar}(t,x)$ satisfies the final value problem
\begin{align}\label{semi-classical-form}
\begin{cases}
\displaystyle
\hbar \frac{\partial \bar{u}^{\hbar}}{\partial t}(t,p)=-\frac{\hbar^2}{2}\triangle \bar{u}^{\hbar}(t,p)+\frac{1}{2}p^2 \bar{u}^{\hbar}(t,p),& (t,p)\in [0,1)\times \mathbb{R}\\
\bar{u}^{\hbar}(1,p)=\exp \left( -p^2/\hbar \right) ,&p\in\mathbb{R}.
\end{cases}
\end{align}
From Schilder's large deviation formula (see \cite{Schilder-1966}), $\bar{u}^{\hbar}(t,x)$ admits the asymptotic expansion
\begin{align}\label{eq:connection-with-EQM}
\bar{u}^{\hbar}(t,p)=\exp\left( -\frac{1}{\hbar}\left(
 \int_t^1\left(\frac{1}{2} |\bar{q}'(\tau)|^2+\frac{1}{2} |\bar{q}(\tau)+p|^2\right)d\tau+|\bar{q}(1)+p|^2+o(1)\right) \right)
\end{align}
 as $\hbar\to 0$,
 where $\bar{q}(\tau):=\bar{\phi}(\tau-t)$ and $\bar{\phi}$ is a minimizer of
$$\inf_{\phi\in A_0 ( [0,1-t] ) }\left(
 \int_0^{1-t}\left(\frac{1}{2} |\phi'(s)|^2+\frac{1}{2} |\phi(s)+p|^2\right)ds+|\phi(1-t)+p|^2\right)
$$ 
 with $A_0 ( [0,1-t] ) $ being the space of all absolutely continuous functions on $[0,1-t]$ having initial value $0.$ In terms of $\bar{q}$ itself, the variational problem can therefore be rewritten as on $[t,1]:$
$$\inf_{q\in A_0 ( [t,1] ) }\left(
 \int_t^1\left(\frac{1}{2} |q'(s)|^2+\frac{1}{2} |q(s)+p|^2\right)ds+|q(1)+p|^2\right).
$$
If $\bar{\phi}$ is smooth, then it satisfies an ordinary differential equation
$$\bar{\phi}''(s)-(\bar{\phi}(s)+p)=0,\,\,\,s\in(0,1-t),\quad \bar{\phi}(0)=0\text{ and }\bar{\phi}'(1-t)+2(\bar{\phi}(1-t)+p)=0.$$
It is again equivalent but more appropriate for us to rewrite this in terms of $\bar{q}$ as
$$\bar{q}''(s)-(\bar{q}(s)+p)=0,\,\,\,s\in(t,1),\quad \bar{q}(t)=0\text{ and }\bar{q}'(1)+2(\bar{q}(1)+p)=0.$$
We note that \eqref{eq:connection-with-EQM} has a quite similar expression as the semiclassical limit (5.38) in \cite{Kolsrud-Zambrini-1991}. Actually, we can match exactly these two expressions by investigating the exact order $o(1)$ in \eqref{eq:connection-with-EQM}. By applying precise large deviations of \cite{Schilder-1966} to $\sqrt{\hbar}W,$ we get the expansion
\begin{eqnarray} 
\label{eq:connection-with-EQM-precise}
\bar{u}^{\hbar}(t,p) & = & (1+o(1)) \bar{K}(t) 
\\ 
\nonumber 
 & & \times \exp\left(
 -\frac{1}{\hbar}\left(
 \int_t^1\left(\frac{1}{2} |\bar{q}'(\tau)|^2+\frac{1}{2} |\bar{q}(\tau)+p|^2\right)d\tau+|(\bar{q}(1)+p|^2\right)
 \right)
\end{eqnarray} 
 as $\hbar\to 0$,
 where $\bar{K}(t)=\mathbb{E} \left[
 \exp \left(
 -\frac{1}{2} \int_0^{1-t}W_s^2\,\,ds-W_{1-t}^2\right) \right].$
 In order to compare \eqref{eq:connection-with-EQM-precise} with
 (6.8)-(6.9) in \cite{Kolsrud-Zambrini-1991},
 we note that $\bar{K}(t)=(2\pi F(t))^{-1/2}$ where
$$
 F(t)=\frac{1}{2\pi}\left(\cosh(1-t)+2\sinh(1-t)\right)$$
 solving the following ordinary differential equation on $[t,1]$ (cf. (1.9.3)
 page~168 in \cite{Borodin-Salminen-2002})
$$F''(\tau)=F(\tau),\quad
 \tau \in [t,1], \quad F(1)=1/(2\pi),\quad F'(1)=-1/\pi.$$
In particular, these final boundary conditions satisfy the condition (6.9) of \cite{Kolsrud-Zambrini-1991}: $|F'(1)|+|F(1)|>0.$ This means that the associated (forward) semiclassical Bernstein diffusion $Z$ of \cite{Kolsrud-Zambrini-1991} is absolutely continuous with respect to our $\xi^\hbar=\sqrt{\hbar}W$ and with linear drift $B(z,\tau)=F'(\tau)F^{-1}(\tau)z$ (the constant $\delta,$ there, is zero). The semiclassical analysis done here for \eqref{semi-classical-form} is valid for Gaussian final boundary condition $\bar{u}^{\hbar}(1,p).$

Equation \eqref{semi-classical-form} is the forward description (involving the usual increasing filtration representing the past information about the system) of the semiclassical expansion, and we can similarly consider the following backward description
\begin{align}\label{semi-classical-form-prime}
\begin{cases}
\displaystyle
\hbar \frac{\partial u^{\hbar}_*}{\partial t}(t,p)=\frac{\hbar^2}{2}\triangle u^{\hbar}_*(t,p)-\frac{1}{2}p^2 u^{\hbar}_*(t,p),& (t,p)\in (0,1]\times \mathbb{R}
\\
\displaystyle
u^{\hbar}(0,p)=\exp \left( -p^2/\hbar \right),&p\in\mathbb{R}
\end{cases}
\end{align}
associated with a decreasing filtration. For simplicity, the same Gaussian boundary condition has been chosen. In this case, the solution $u^{\hbar}_*(t,p)$ admits the expansion
\begin{eqnarray}\label{eq:connection-with-EQM-precise-prime}
\lefteqn{
u^{\hbar}_*(t,p)=K_*(t)\cdot(1+o(1))}
\\
\nonumber
 & & \times \exp\left(
 -\frac{1}{\hbar}\left( 
 \int_0^t\left(\frac{1}{2} |q_*'(\tau)|^2+\frac{1}{2} |q_*(\tau)+p|^2\right)d\tau+|q_*(t)+p|^2\right) 
 \right)
\end{eqnarray}
where $q_*$ is the minimizer of
$$
\inf_{q\in A_0 ( [0,t] ) }\left( 
\int_0^t\left(\frac{1}{2} |q'(s)|^2+\frac{1}{2} |q(s)+p|^2\right)ds+|q(t)+p|^2\right), 
$$
 and the coefficient $K_*(t)=\mathbb{E} \left[
 \exp \left( -\frac{1}{2} \int_0^tW_s^2\,\,ds-W_t^2\right)\right].$
 Then we have $K_*(t)=(2\pi F_*(t))^{-1/2},$ where $F_*$ solves the same ODE as $F$ before but with initial boundary conditions on $[0,t]: F_*(0)=1/(2\pi)$ and $F_*'(0)=1/\pi.$ Then $F_*(\tau)=\frac{1}{2\pi}(\cosh\tau+2\sinh\tau).$ The same semiclassical diffusion $Z$ as before, now considered with respect to a decreasing, or backward, filtration has a drift $B_*(z,\tau)=F_*'(\tau)F_*^{-1}(\tau)z.$ Again, the whole manifold of semiclassical diffusions follow from the analysis of general Gaussian initial condition (cf. \cite{Kolsrud-Zambrini-1991}). The special Gaussian $Z$ defined via \eqref{eq:connection-with-EQM-precise} and \eqref{eq:connection-with-EQM-precise-prime} is of zero mean and covariance $c(\tau)$ proportional to $F_*(\tau)F(\tau)$, $0\leq \tau\leq1.$
\subsection{Pure jump Hamiltonian}
 The operator $H=- \hbar^2 \triangle / 2 +p^2/2$ of equation \eqref{semi-classical-form} is called \textit{Harmonic oscillator} Hamiltonian expressed here, in this special case where configuration and momentum play a completely symmetric role, in terms of the momentum variable. We can generalize \eqref{semi-classical-form} to include more general Hamiltonians and derive precise large deviations as \eqref{eq:connection-with-EQM-precise} having similar connections with the semiclassical limits of (Euclidean) quantum physics, see \cite{Kolsrud-Zambrini-1991}, \cite{Wentzell-LD-Markov-Processes-1986} and \cite{Yang-Annals-2012+}. In particular, we illustrate below an example with pure jump processes as mentioned in Section \ref{subsec:an-illustration}. Consider the following partial integro-differential equation with final condition
\begin{align}\label{semi-classical-form-jump}
\begin{cases}
\displaystyle
\hbar \frac{\partial \bar{v}^{\hbar}}{\partial t}(t,p)=-\int_{\mathbb{R}} 
 ( \bar{v}^{\hbar}(t,p+\hbar k)-\bar{v}^{\hbar}(t,p)) 
 \mu(dk)+U(p) \bar{v}^{\hbar}(t,p),& (t,p)\in [0,1)\times \mathbb{R}\\
\bar{v}^{\hbar}(1,p)=\exp \left( -1/\hbar \right), \quad p\in\mathbb{R}. &
\end{cases}
\end{align}
If we consider a special potential $V(x)=1-\cos(\alpha x)$ for some $\alpha>0,$ this case corresponds to $\mu(dk)=\frac{1}{2}( \delta_\alpha(dk)+\delta_{-\alpha}(dk)).$ 
 For a particular $U(p)=p^2-p,$ it has been proved in \cite{Yang-Annals-2012+} that $\bar{v}^{\hbar}$ has the asymptotics
\begin{align}\label{eq:connection-with-EQM-precise-jump}
\bar{v}^{\hbar}(t,p)=\bar{P}(t)\cdot(1+o(1))\cdot\exp\left( -\frac{1}{\hbar} 
 \left( \int_t^1L(\bar{z}'(s),\bar{z}(s))ds+1\right)  
\right)
\end{align}
 as $\hbar\to 0$,
 where $L$ is defined in Section \ref{subsec:an-illustration} with
\begin{align}\label{S-end}
S(z)=\int_t^{1}\left( (z'(s)/\alpha)\ln \left(z'(s)/\alpha+\sqrt{\left(z'(s)/\alpha\right)^2+1}\right)+1-\sqrt{\left(z'(s)/\alpha\right)^2+1}\right)ds
\end{align}
 for absolutely continuous $z(s),$ the function $\bar{z}(\tau)$ is the unique minimizer of
$$\int_t^{1}L(z'(s),z(s))ds$$
 over $A_0 ( [t,1]),$ and $\bar{P}(t)=\mathbb{E}\left[
 \exp \left( -\int_t^1\bar{\varsigma}_s^2\,\,ds \right)\right]$
 with
$$\bar{\varsigma}_s = \alpha
 \sqrt{\frac{e^{\bar{\varrho}(s-t)}+e^{-\bar{\varrho}(s-t)} }{2} } W_{s-t}$$ and
$$\bar{\varrho}(\tau)=\ln \left(\bar{z}'(\tau)/\alpha+\sqrt{(\bar{z}'(\tau)/\alpha)^2+1}\right).$$
We further remark that the unique minimizer $\bar{z}$ satisfies the following ordinary differential equation
$$\bar{z}''(s)- ( 2(\bar{z}(s)+p)-1) \cdot\sqrt{|\bar{z}'(s)|^2+\alpha^2}=0,\quad s\in(t,1),\quad \bar{z}(t)=0\text{ and }\bar{z}'(1)=0.$$
Again \eqref{eq:connection-with-EQM-precise-jump} is the forward description of the semiclassical expansion. Let us now consider the backward description by using
\begin{align}\label{semi-classical-form-jump-backward}
\begin{cases}
 \displaystyle
 \hbar \frac{\partial v_{*}^{\hbar}}{\partial t}(t,p)=\int_{\mathbb{R}} 
 ( v_{*}^{\hbar}(t,p+\hbar k)-v_{*}^{\hbar}(t,p) ) 
 \mu(dk)-U(p) v_{*}^{\hbar}(t,p),& (t,p)\in (0,1]\times \mathbb{R}\\
\displaystyle
v_{*}^{\hbar}(0,p)=\exp \left( -1/\hbar \right),&p\in\mathbb{R}.
\end{cases}
\end{align}
In this case, we have
\begin{align}\label{eq:connection-with-EQM-precise-jump-forward}
v_{*}^{\hbar}(t,p)=P_*(t)\cdot(1+o(1))\cdot\exp\left( -\frac{1}{\hbar}\left( 
 \int_0^tL(z_*'(s),z_*(s))ds+1\right) 
 \right)
\end{align}
 as $\hbar\to 0$,
 where $z_*$ is the unique minimizer of $\int_0^tL(z'(s),z(s))ds$ over
 $A_0 ( [0,t] )$ (of course the action functional $S$ defined by \eqref{S-end} now has the integral interval $[0,t]$), and
$$P_*(t)=\mathbb{E} \left[
 \exp \left( -\int_0^{t}\varsigma_s^2\,\,ds \right) \right]$$
 with
$\varsigma_s= \alpha \sqrt{ (e^{\varrho_*(s)}+e^{-\varrho_*(s)})/2}W_{s}$ and
$$\varrho_*(\tau)=\ln \left(z_*'(\tau)/\alpha+\sqrt{|z_*'(\tau)/\alpha |^2+1}\right).$$

\subsection{Precise asymptotics for drift terms}
In \cite{chungzambrini}, Bernstein processes were studied as stochastic deformation related to Feynman path integral. These processes were first introduced in \cite{bernstein}, and now are also called \textit{local Markov} processes, \textit{two-sided} Markov processes or \textit{reciprocal} process in literature, see \cite{jamison}. In configuration representations, basically they are diffusion processes over a finite time interval with prescribed initial and final distributions. In momentum respresentions, the drift parts (terms) can be expressed as $\hbar\nabla u^{\hbar}_*(t,p)/u^{\hbar}_*(t,p)$ with $u^{\hbar}_*(t,p)$ defined in \eqref{eq:momen-equation}. In the free case (that is, the potential is identically zero), the drift has an explicit form which is independent of $\hbar$ (cf. Section 5.1 in \cite{chungzambrini}). Our goal in this section is to derive precise asymptotics as $\hbar\to 0$ for the drift terms in both configuration and momentum representations. In order to have explicit comparisons with the results in \cite{chungzambrini}, here we adopt some new notations which are slightly different from those in previous sections. For instance, the solution $u^{\hbar}_*(t,p)$ is now denoted as $\eta^{\hbar}_*(t,q)$ or $\tilde{\eta}^{\hbar}_*(t,p),$ and the potential is still denoted by $V$ (with different expressions) which is fully explained each time when it is used.

\subsubsection{Configuration representation Hamiltonian}
 We consider a Hamiltonian
 $H=- \hbar^2 \triangle / 2+V$ in configuration representation,
 where $V$ is a potential in Kato's class.
 The backward description of the semiclassical expansion is
\begin{align}\label{eq:PDE-end}
\begin{cases}
\displaystyle
\hbar \frac{\partial \eta^{\hbar}_*}{\partial t}(t,q)=\frac{\hbar^2}{2}\triangle \eta^{\hbar}_*(t,q)-V(q)\eta^{\hbar}_*(t,q),& (t,q)\in (0,\infty)\times \mathbb{R}\\
\eta^{\hbar}_*(0,q)=\frac{1}{\sqrt{2\pi \hbar}}e^{-q^2/(2\hbar)}, \quad
 q\in\mathbb{R}. &
\end{cases}
\end{align}
We aim at finding precise asymptotics for $\hbar\nabla\eta^{\hbar}_*(t,q)/\eta^{\hbar}_*(t,q)$ as $\hbar\to 0.$ To this end, we introduce the space $C_0 ( [0,t] )$ which consists of continuous functions with zero initial value, $C_0^1 ( [0,t] )$ whose elements are in $C_0 ( [0,t] )$ and continuously differentiable, and three functionals
\begin{align*}
& F(\phi)=- \left( \int_0^tV(\phi(s)+q)ds+\frac{|\phi(t)+q|^2}{2}\right),
\\
& G(\phi)=\int_0^tV'(\phi(s)+q)ds+(\phi(t)+q),
\\
& S(\phi)=\frac{1}{2}\int_0^t\phi'(s)^2ds,
\end{align*}
 for absolutely continuous $\phi \in C_0 ( [0,t] )$,
 and $S(\phi)=\infty$ otherwise. Now we state the main result of this section.
\begin{prop}\label{th:configuration}
 Assume that
\begin{description} 
\item{(A.1)}
 the potential $V(\cdot):\mathbb{R}\rightarrow\mathbb{R}$
 is bounded below and has up to fourth order continuous derivatives with $V''(q)\geq0$, and
\item{(A.2)} 
 the supremum of $F-S$ over $C_0 ( [0,t] )$
 is reached uniquely at $\phi_*\in C_0^1( [0,t] )$.
\end{description} 
 Then we have the equivalence
\begin{align*}
\hbar \frac{\nabla\eta^{\hbar}_*(t,q)}{\eta^{\hbar}_*(t,q)}\sim -\left(
 G(\phi_*)+\frac{\bar{K}_1^*}{K_0^*}\hbar+o(\hbar)\right),
\end{align*}
 as $\hbar\to 0,$ where
\begin{align*}
& K_0^*=\mathbb{E} \left[
 \exp\left(\frac{1}{2}F''(\phi_*)(W,W)\right) \right]
\\
& \bar{K}_0^*=G(\phi_*)\cdot K_0^*\\
& \bar{K}_1^*=\mathbb{E}\left[ \exp \left(\frac{1}{2}F''(\phi_*)(W,W)\right) \left( \frac{1}{2}G''(\phi_*)(W,W)+\frac{1}{6}G'(\phi_*)(W)F^{(3)}(\phi_*)(W^{\otimes^3})\right.\right.\\
&\qquad\qquad\qquad\qquad\quad\left.\left.+\frac{1}{24}G(\phi_*)F^{(4)}(\phi_*)(W^{\otimes^4})+\frac{1}{72}G(\phi_*)\left(F^{(3)}(\phi_*)(W^{\otimes^3})\right)^2\right) \right].
\end{align*}
\end{prop}

\begin{remark}
 On the interval $[0,t],$ the unique maximizer $\phi_*$ satisfies the
 ordinary differential equation
$$\phi_*''(\tau)-V'(\phi_*(\tau)+q)=0, \text{ with }\phi_*(0)=0, \phi_*'(t)=-(\phi_*(t)+q),
$$
 with in particular $G(\phi_*) =-\phi_*'(0).$
 In order to express the dependence of $G(\phi_*)$ on $q$ and the time $t$
 we can use
 the transform $\Phi_*(\tau)=\phi_*(t-\tau)+\tau\cdot q$
 and in this case
$$\Phi_*''(\tau)-V'(\Phi_*(\tau)-\tau\cdot q+q)=0, \text{ with }\Phi_*(t)=qt, \Phi_*'(0)=\Phi_*(0)+q$$
from which $G(\phi_*)=-\phi_*'(0)=\Phi_*'(t)-q.$
\end{remark}
\begin{remark}
In the formulation of Proposition \ref{th:configuration}, all functional derivatives are in the sense of Fr\'{e}chet derivative. For instance,
$$
 F'(\phi_*)(W)=-\left(
 \int_0^tV'(\phi_*(s)+q)W_sds+(\phi_*(t)+q)W_t\right)
$$
 and $F''(\phi_*)(W,W)=-\left(
 \int_0^tV''(\phi_*(s)+q)W_s^2ds+W_t^2\right)$.
\end{remark}

\begin{proof}[Proof of Proposition \ref{th:configuration}]

The solution to \eqref{eq:PDE-end} can be written by a Feynman-Kac formula as
\begin{align}\label{eq:Feyn-Kac-formula-end}
\eta^{\hbar}_*(t,q)=\frac{1}{\sqrt{2\pi \hbar}}\mathbb{E}
 \left[
 \exp\left( \frac{1}{\hbar}F(\sqrt{\hbar}W)\right) \right],
\end{align}
where the functional $F$ is defined in Proposition \ref{th:configuration}. Then the derivative of $\eta^{\hbar}_*(t,q)$ with respect to $q$ is written as
\begin{align}\label{eq:derivative-Feyn-Kac-formula-end}
\nabla\eta^{\hbar}_*(t,q)=-\frac{1}{\hbar}\cdot\frac{1}{\sqrt{2\pi \hbar}}\mathbb{E}\left[G(\sqrt{\hbar}W)\cdot\exp \left(
 \frac{1}{\hbar}F(\sqrt{\hbar}W)\right) \right]
\end{align}
with the functional $G$ as above. From precise asymptotics for large deviations of integral forms (cf. \cite{Wentzell-LD-Markov-Processes-1986} and \cite{Yang-Annals-2012+}), expansions for \eqref{eq:Feyn-Kac-formula-end} and \eqref{eq:derivative-Feyn-Kac-formula-end} can be proved as follows, respectively,
\begin{equation}\label{expansion-Feyn-Kac-formula-end}
\begin{aligned}
\eta^{\hbar}_*(t,q)&=\frac{1}{\sqrt{2\pi \hbar}}\mathbb{E}
 \left[
 \exp\left(
 \frac{1}{\hbar}F(\sqrt{\hbar}W)\right) \right]
\\
&=\frac{1}{\sqrt{2\pi \hbar}}\exp\left(
\frac{1}{\hbar}\left(
 F(\phi_*)-S(\phi_*)\right)
 \right)
 \cdot\left( K_0^*+o(1)\right)
\end{aligned}
\end{equation}
 and
\begin{equation}\label{expansion-deriv-Feyn-Kac-formula-end}
\begin{aligned}
\nabla\eta^{\hbar}_*(t,q)&=-\frac{1}{\hbar}\cdot\frac{1}{\sqrt{2\pi \hbar}}\mathbb{E}\left[G(\sqrt{\hbar}W)\cdot\exp\left(
 \frac{1}{\hbar}F(\sqrt{\hbar}W)\right)
 \right]\\
&=-\frac{1}{\hbar}\cdot\frac{1}{\sqrt{2\pi \hbar}}\exp\left(
 \frac{1}{\hbar}\left(
 F(\phi_*)-S(\phi_*)\right)
 \right)
 \cdot\left( 
 \bar{K}_0^*+\bar{K}_1^*\hbar+o(\hbar)\right),
\end{aligned}
\end{equation}
 as $\hbar\to 0$.
 Combining \eqref{expansion-Feyn-Kac-formula-end} and \eqref{expansion-deriv-Feyn-Kac-formula-end}, it follows that
\begin{align*}
 \hbar \frac{\nabla\eta^{\hbar}_*(t,q)}{\eta^{\hbar}_*(t,q)}&\sim\frac{-
 \exp\left( \left(
 F(\phi_*)-S(\phi_*)\right) / \hbar
 \right)
 \cdot\left( 
 \bar{K}_0^*+\bar{K}_1^*\hbar+o(\hbar)\right) 
 }{\exp\left(
 \frac{1}{\hbar}\left( 
 F(\phi_*)-S(\phi_*)\right) 
 \right)
 \cdot\left( K_0^*+o(1)\right) }\\
&\sim -\left(
 G(\phi_*)+\frac{\bar{K}_1^*}{K_0^*}\hbar+o(\hbar)\right),
\end{align*}
 as $\hbar\to 0$.
\end{proof}
\subsubsection{Momentum representation Hamiltonian}
 Consider the complex-valued potential
$$
 V(q)=ibq+\frac{\sigma^2q^2}{2}+\int_{\mathbb{R}\setminus\{0\}}\left(
 e^{-iqk}-1+iqk {\bf 1}_{\{|k|\leq1\}}\right)
 \nu(dk)
$$
 which yields the pseudo-differential operator
\begin{align*}
V(i\hbar \nabla)u(p)=-\hbar b u'(p)-\hbar^2 \frac{\sigma^2}{2}u''(p)-\int_{\mathbb{R}\setminus\{0\}}\left(
 u(p+\hbar k)-u(p)\right)
\nu(dk)
\end{align*}
 where the L\'{e}vy measure $\nu(dk)$ is assumed to be
 symmetric under the time reversal $k\to -k$.
\\

 The resulting Hamiltonian $\hat{H}= p^2/2 +V(i\hbar \nabla)$ is the momentum representation of
 $H=- \hbar^2 \triangle / 2 +V.$ Then the backward description is the following partial integral-differential equation,
\begin{align}\label{eq:momen-equation-end}
\begin{cases}
\displaystyle
\hbar \frac{\partial \tilde{\eta}^{\hbar}_*}{\partial t}(t,p)=-\frac{1}{2}p^2 \tilde{\eta}^{\hbar}_*(t,p)-V(i\hbar \nabla)\tilde{\eta}^{\hbar}_*(t,p),& (t,p)\in (0,\infty)\times \mathbb{R}\\
\tilde{\eta}^{\hbar}_*(0,p)=\frac{1}{\sqrt{2\pi \hbar}}e^{-p^2/(2\hbar)},
 \quad p\in\mathbb{R}. &
\end{cases}
\end{align}
Because of the jumps, in this section we consider the space $D_0 ([0,t]).$ In order to obtain the precise asymptotics for $\hbar\nabla\tilde{\eta}^{\hbar}_*(t,q)/\tilde{\eta}^{\hbar}_*(t,q)$ as $\hbar\to 0,$ we define three functionals over $D_0 ([0,t]):$
\begin{align*}
& \tilde{F}(\phi)= - \frac{1}{2}
 \int_0^t |\phi(s)+p|^2 ds
 - \frac{1}{2} |\phi(t)+p|^2,
\\
& \tilde{G}(\phi) =
 \phi(t)+p + \int_0^t(\phi(s)+p)ds,
\\
& \tilde{S}(\phi)=\int_0^tL_0(\phi'(s))ds,
\end{align*}
 for absolutely continuous $\phi,$ where $L_0(u)=\sup_{x\in \mathbb{R}} 
 ( xu-H_0(x) )$ and
$$
 H_0(x)=bx+\frac{\sigma^2x^2}{2}+\int_{\mathbb{R}\setminus\{0\}} (e^{xk}-1)\nu(dk),
$$
 and $S(\phi)=\infty$ otherwise.
\\

The operator $V(i\hbar \nabla)$ is associated with a family of L\'{e}vy processes which satisfies a large deviation principle under
 suitable assumptions on $H_0$,
 cf. Conditions~(A)-(E) in
 Sections 3.1 and 3.2 of \cite{Wentzell-LD-Markov-Processes-1986},
 which are assumed to hold in the
 following proposition.
\begin{prop}\label{th:momentum}
If the supremum of $\tilde{F}-\tilde{S}$ over $D_0 ([0,t])$ is reached uniquely at $\tilde{\phi}_*\in C_0^1 ( [0,t] ),$ then following precise asymptotics holds:
\begin{align*}
\hbar \frac{\nabla\tilde{\eta}^{\hbar}_*(t,p)}{\tilde{\eta}^{\hbar}_*(t,p)}\sim
 - \tilde{G}(\tilde{\phi}_*) + o(\hbar),
\end{align*}
 as $\hbar\to 0$.
\end{prop}
\begin{remark}
The unique maximizer $\tilde{\phi}_*$ satisfies an ordinary differential equation on $[0,t]:$
$$\frac{d}{d\tau}\left(L_0'(u)|_{u=\tilde{\phi}_*'(\tau)}\right)=\tilde{\phi}_*(\tau)+p, \text{ with } \tilde{\phi}_*(0)=0, L_0'(u)|_{u=\tilde{\phi}_*'(t)}=-(\tilde{\phi}_*(t)+p),$$
from which it follows $\tilde{G}(\tilde{\phi}_*)=-L_0'(u)|_{u=\tilde{\phi}_*'(0)}.$
\end{remark}

\begin{proof}[Proof of Proposition \ref{th:momentum}]
The solution to \eqref{eq:momen-equation-end} is written as a Feynman-Kac formula
\begin{align}\label{eq:Feyn-Kac-formula-momen-end}
\tilde{\eta}^{\hbar}_*(t,p)=\frac{1}{\sqrt{2\pi \hbar}}\mathbb{E}^\hbar
 \left[ \exp\left(
 \frac{1}{\hbar}\tilde{F}(\xi^{\hbar})\right) \right],
\end{align}
 where $( \xi^{\hbar}_t )_{t\in \real_+}$
 defined on $(\Omega,\mathcal{F},\mathbb{P}^{\hbar})$ is a real-valued L\'{e}vy process for each fixed $\hbar$ whose characteristic function is given by the L\'{e}vy-Khintchine formula
\begin{align*}
\mathbb{E}^{\hbar}\left[\exp\left(
 -\frac{i}{\hbar}\cdot x\cdot \xi^{\hbar}_t\right)
 \right]=\exp\left( -\frac{t}{\hbar}V(x)\right)
\end{align*}
with $V(x)=ibx+\frac{\sigma^2}{2}x^2-\int_{\mathbb{R}\setminus\{0\}}\left(
 e^{-ix k}-1\right) \nu(dk).$ The derivative of $\tilde{\eta}^{\hbar}_*(t,p)$ with respect to $p$ is
\begin{align}\label{eq:derivative-Feyn-Kac-formula-momen-end}
\nabla\tilde{\eta}^{\hbar}_*(t,p)=-\frac{1}{\hbar}\cdot\frac{1}{\sqrt{2\pi \hbar}}\mathbb{E}^\hbar\left[\tilde{G}(\xi^{\hbar})\cdot\exp\left(
 \frac{1}{\hbar}\tilde{F}(\xi^{\hbar})\right)
 \right].
\end{align}
It is again from \cite{Wentzell-LD-Markov-Processes-1986} and \cite{Yang-Annals-2012+} that expansions for \eqref{eq:Feyn-Kac-formula-momen-end} and \eqref{eq:derivative-Feyn-Kac-formula-momen-end} are
\begin{equation}\label{expansion-Feyn-Kac-formula-momen-end}
\begin{aligned}
\tilde{\eta}^{\hbar}_*(t,p)&=\frac{1}{\sqrt{2\pi \hbar}}\mathbb{E}^\hbar
 \left[
 \exp
 \left( \frac{1}{\hbar}\tilde{F}(\xi^{\hbar})\right) \right]
\\
&=\frac{1}{\sqrt{2\pi \hbar}}\exp\left(
 \frac{1}{\hbar}\left(
 \tilde{F}(\tilde{\phi}_*)-\tilde{S}(\tilde{\phi}_*)\right)
 \right)
 \cdot\left( K_0+o(1)\right)
\end{aligned}
\end{equation}
and
\begin{equation}\label{expansion-deriv-Feyn-Kac-formula-momen}
\begin{aligned}
\nabla\tilde{\eta}^{\hbar}_*(t,p)&=-\frac{1}{\hbar}\cdot\frac{1}{\sqrt{2\pi \hbar}}\mathbb{E}^\hbar\left[\tilde{G}(\xi^{\hbar})\cdot\exp\left(
 \frac{1}{\hbar}\tilde{F}(\xi^{\hbar})\right)
 \right]\\
&=-\frac{1}{\hbar}\cdot\frac{1}{\sqrt{2\pi \hbar}}\exp\left(
 \frac{1}{\hbar}\left( \tilde{F}(\tilde{\phi}_*)-\tilde{S}(\tilde{\phi}_*)\right)
 \right)
 \cdot\left(
 \bar{K}_0+\bar{K}_1\hbar+o(\hbar)\right),
\end{aligned}
\end{equation}
where $K_0=\mathbb{E}
 \left[
 \exp\left(
 \frac{1}{2}\tilde{F}''(\tilde{\phi}_*)(\zeta,\zeta)\right) \right] ,$
 $\bar{K}_0=\tilde{G}(\tilde{\phi}_*)\cdot K_0,$
\begin{align*}
\bar{K}_1&=\mathbb{E}\left[ \exp\left(
 \frac{1}{2}\tilde{F}''(\tilde{\phi}_*)(\zeta,\zeta)\right)
 \left(
 \frac{1}{2}\tilde{G}''(\tilde{\phi}_*)(\zeta,\zeta)+\frac{1}{6}\tilde{G}'(\tilde{\phi}_*)(\zeta)\tilde{F}^{(3)}(\tilde{\phi}_*)(\zeta^{\otimes^3})\right.\right.\\
&\qquad\qquad\qquad\qquad\quad\left.\left.+\frac{1}{24}\tilde{G}(\tilde{\phi}_*)\tilde{F}^{(4)}(\tilde{\phi}_*)(\zeta^{\otimes^4})+\frac{1}{72}\tilde{G}(\tilde{\phi}_*)\left(\tilde{F}^{(3)}(\tilde{\phi}_*)(\zeta^{\otimes^3})\right)^2\right)
 \right]=0
\end{align*}
for $\zeta$ being a diffusion process with diffusion coefficient $\sigma^2+\int_{\mathbb{R}\setminus\{0\}} e^{z_*(t)k}k^2\nu(dk)$ and $z_*(t)=L_0'(u)|_{u=\tilde{\phi}_*(t)}.$ It follows from combining \eqref{expansion-Feyn-Kac-formula-momen-end} and \eqref{expansion-deriv-Feyn-Kac-formula-momen} that
\begin{align*}
\hbar \frac{\nabla\tilde{\eta}^{\hbar}_*(t,q)}{\tilde{\eta}^{\hbar}_*(t,q)}&\sim\frac{-\exp\left( \frac{1}{\hbar}\left(
 \tilde{F}(\tilde{\phi}_*)-\tilde{S}(\tilde{\phi}_*)\right)
 \right)
 \cdot\left(
 \bar{K}_0+\bar{K}_1\hbar+o(\hbar)\right)
 }{\exp\left(
 \frac{1}{\hbar}\left( \tilde{F}(\tilde{\phi}_*)-\tilde{S}(\tilde{\phi}_*)\right) \right)
 \cdot\left( K_0+o(1)\right)
 }\sim -\left( \tilde{G}(\tilde{\phi}_*)+o(\hbar)\right)
\end{align*}
 as $\hbar\to 0,$ which completes the proof.
\end{proof}

\footnotesize 

\baselineskip0.5cm

\def\cprime{$'$} \def\polhk#1{\setbox0=\hbox{#1}{\ooalign{\hidewidth
  \lower1.5ex\hbox{`}\hidewidth\crcr\unhbox0}}}
  \def\polhk#1{\setbox0=\hbox{#1}{\ooalign{\hidewidth
  \lower1.5ex\hbox{`}\hidewidth\crcr\unhbox0}}} \def\cprime{$'$}

\end{document}